\documentclass[draft]{amsart}

\usepackage{amssymb}
\usepackage{url}
\usepackage{amscd}

\usepackage[all]{xy}

\def \N{{\mathbb N}}

\def \R{{\mathbb R}}

\def \1{{\mathbb 1}}

\theoremstyle{plain}
\newtheorem{theorem}{Theorem}
\newtheorem{proposition}{Proposition}
\newtheorem{definition}{Definition}
\newtheorem{lemma}{Lemma} 
\newtheorem{corollary}{Corollary}

\theoremstyle{remark}

\begin{document}


\title[On the Bauer's maximum principle]{Extension of the Bauer's maximum principle for compact metrizable sets.}

\author{Mohammed Bachir}

\address{Laboratoire SAMM 4543, Universit\'e Paris 1 Panth\'eon-Sorbonne\\
Centre P.M.F. 90 rue Tolbiac\\
75634 Paris cedex 13\\
France}

\email{Mohammed.Bachir@univ-paris1.fr}



\begin{abstract}
Let $X$ be a nonempty convex compact subset of some Hausdorff locally convex topological vector space $S$. The well know Bauer's maximum principle stats that every convex upper semi-continuous function from $X$ into $\R$ attains its maximum at some extremal point of $X$. We give some extensions of this result when $X$ is assumed to be compact metrizable. We prove that the set of all convex upper semi-continuous functions attaining there maximum at exactly one extremal point of $X$ is a $G_\delta$ dense subset of the space of all convex upper semi-continuous functions equipped with a metric compatible with the uniform convergence.
%
\end{abstract}

\maketitle


\newcommand\sfrac[2]{{#1/#2}}

\newcommand\cont{\operatorname{cont}}
\newcommand\diff{\operatorname{diff}}


\section{Introduction}

Let $X$ be a nonempty convex compact subset of some Hausdorff locally convex topological vector space $S$. By $\textnormal{Aff}(X)$ we denote the space of all affine (i.e. $\varphi(\lambda x+(1-\lambda)y)=\lambda \varphi(x)+(1-\lambda)\varphi(y)$ for all $x, y\in X$ and $\lambda\in [0,1]$), continuous functions from $X$ into $\R$. 
Let $K$ be a nonempty subset of $X$ (not necessarily convex), a point $x\in K$ is said to be an affine exposed point of $K$, if there exists $\varphi\in \textnormal{Aff}(X)$ such that $\varphi$ attains its unique maximum at $x$. The set of all affine exposed points of $K$ is denoted by $\textnormal{AExp}(K)$. It is easy to see that $\textnormal{AExp}(K)\subset \textnormal{Ext}(K)$, where $\textnormal{Ext}(K)$ is the set of all extremal points of $K$, but this inclusion is strict in general (see \cite{Ba} for more details about affine exposed points). Recall that a point $x$ of a nonempty subset $C$ of $S$ is extremal in $C$, if and only if :
$y, z \in C, 0<\alpha <1 ;\hspace{2mm} x = \alpha y+(1-\alpha)z \Longrightarrow x = y = z.$
\vskip2mm
If $K$ is a non empty compact subset of $X$ and $f: K\to \R$ is upper semi-continuous, by $K_{\max}(f)$, we denote the following non empty closed subset of $K$:
$$\emptyset\neq K_{\max}(f):=\lbrace x\in K: f(x)=\max_K f\rbrace.$$
By $\Sigma(X)$ we denote the space of all convex upper semi-continuous functions from $X$ into $\R$. Let $\mathcal{B}$ be a subset of $\Sigma(X)$. We say that $\mathcal{B}$ is $\textnormal{Aff}(X)$-stable if and only if  
$$(f, \varphi)\in \mathcal{B}\times  \textnormal{Aff}(X)\Longrightarrow f+\varphi \in\mathcal{B}.$$ 
The space $\Sigma(X)$ of all convex upper semi-continuous (resp. the space of all convex continuous, the space of all affine continuous) functions from $X$ into $\R$ is $\textnormal{Aff}(X)$-stable.  These spaces are equipped with the following metric:
$$\rho_{\infty}(f,g):=\sup_{x\in X} \frac{|f(x)-g(x)|}{1+|f(x)-g(x)|}.$$
Recall that a function $f$ is said to be strictly quasi-convex on a convex set $X$ if for all $y, z\in X$ and all $\lambda \in ]0,1[$, we have $f(\lambda y+(1-\lambda)z)< \max(f(y),f(z))$.
\vskip5mm
We stat now the Bauer's maximum principle which applies for convex compact not necessarily metrizable subsets of a Hausdorff locally convex topological vector space. The classical proof of Bauer's maximum principle is based on Zorn's lemma.
\begin{theorem} \label{Bau} \textnormal{(Bauer's maximum principle, \cite{Bau})} Let $S$ be a Hausdorff locally convex topological vector space and $X$ a nonempty convex compact subset of $S$. Let $f: X\to \R$ be a convex upper semi-continuous function. Then, $$X_{\max}(f)\cap \textnormal{Ext}(X)\neq \emptyset.$$
\end{theorem}
The aim of this note is to give with a simple proof and without using Zorn's lemma, some extensions of Bauer's maximum principle in the compact metrizable framework. Indeed, we prove that if $X$ is metrizable, then :

$(1)$ Let $\mathcal{B}$ be an $\textnormal{Aff}(X)$-stable subset of the space of all convex upper semi-continuous functions from $X$ into $\R$. Then, for any nonempty compact (not necessarily convex) subset $K$ of $X$ there exists a $G_\delta$ dense subset $\mathcal{G}$ of $(\mathcal{B},\rho_{\infty})$ such that for every function $f\in \mathcal{G}$ we have that $\textnormal{Ext}(K)\cap K_{\max}(f)$ is a singleton.

$(2)$ for every nonempty familly $(f_i)_{i\in I}$ ($I$ a nonempty set) of convex upper semi-continuous (resp. of strictly quasi-convex upper semi-continuous) functions on $X$ and every nonempty closed (not necessarily convex) subset $K$ of $X$, we have that 
\begin{eqnarray*}
\cap_{i\in I} K_{\max}(f_i)\neq \emptyset&\Longrightarrow&\textnormal{AExp}(\cap_{i\in I} K_{\max}(f_i))\cap \textnormal{Ext}(K) \neq \emptyset\\
&\Longrightarrow&\cap_{i\in I} K_{\max}(f_i)\cap \textnormal{Ext}(K) \neq \emptyset.
\end{eqnarray*}
In particular, there exists a commun extremal point $e\in\textnormal{Ext}(K)$ at which $f_i$ attains its maximum over $K$, for every $i\in I$. We recover immediately the classical Bauer's maximum principle (in the metrizable case) by taking $K=X$ and $I$ a singleton. As an immediat consequence we have: for every $x\in X$, let $\Omega_x$ be the following nonempty closed convex cone of $(\Sigma(X),\rho_{\infty})$ 
 $$\Omega_x:=\lbrace f\in \Sigma(X): f(x)=\max_X f \rbrace.$$
Then, there exists a commun extremal point $e\in \textnormal{Ext}(X)$ such that every function $f\in \Omega_x$ attains its maximum on $X$ at $e$.

 The main results of this note, are established in the more general context of $\Phi$-convexity (Theorem \ref{Ba}) in the sprit of the works by K. Fan \cite{Ky}, M. W. Grossman \cite{Gr} and B. D. Khanh \cite{Kh}. Note that every compact subset of Fr\'echet space or of metrizable Hausdorff locally convex topological vector space $S$ is of course metrizable, but the class of Hausdorff locally convex topological vector space $S$ in which every compact subset is metrizable, is more larger (see for instance the paper of B. Cascales and J. Orihuela in \cite{CO}). 
\vskip5mm
The proofs of our results are consequences of a new variational principle established recently in \cite{Ba} and does not use Zorn's lemma. Note that varitional principle \cite[Lemma 3.]{Ba} that we will use is similar to that of Deville-Godefroy-Zizler in \cite{DGZ} and Deville-Revalski in \cite{DR}, it applies to compact metrizable sets but the interest is that it does not use the existence of a bump function. This will allow us to work for example, with the space of affine continuous functions defined on convex compact metrizable set or the space of harmonic functions defined on open connexe set of $\R^n$, which has no bump functions.

\section{Bauer's maximum principle.}
Let $S$ be any nonempty set, $\Phi$ a family of real valued functions on $S$. A subset $X\subset S$ is said to be $\Phi$-convex if $X=S$ or there exists a nonempty set $I$, such that $$X=\cap_{i\in I} \lbrace x\in S: \varphi_i(x)\leq \lambda_i\rbrace,$$
where $\varphi_i\in \Phi$ and $\lambda_i\in \R$ for all $i\in I$. For a nonempty set $A\subset S$, the intersection of all $\Phi$-convex subset of $S$ containing $A$ is said to be the $\Phi$-convex hull of $A$. By $\textnormal{conv}_\Phi(A)$, we denote the $\Phi$-convex hull of $A$.

\noindent Let $a, x, y \in S$, we say that $a$ is $\Phi$-between $x$ and $y$, if

$$(\varphi \in \Phi, \varphi(x)\leq\varphi(a), \varphi(y)\leq\varphi(a))\Longrightarrow (\varphi(a)=\varphi(x)=\varphi(y)).$$
Let $\emptyset\neq A \subset B\subset S$. The set $A$ is said to be $\Phi$-extremal subset of $B$, if 
$$(a\in A, a \textnormal{ is } \Phi\textnormal{-between the points } x,y\in B)\Longrightarrow (x\in A, y\in A).$$
If $A$ is a singleton $A=\lbrace a \rbrace$, we say that $a$ is $\Phi$-extremal point of $B$. The set of all $\Phi$-extremal points of a nonempty set $A$ will be denoted by $\Phi\textnormal{Ext}(A)$.
\vskip5mm
When $S$ is a Hausdorff locally convex topological vector space and $\Phi=S^*$ is the topological dual of $S$, then the $\Phi$-extremal points of a set coincides with the classical extremal points (see \cite[Proposition 2.]{Kh} and \cite{Ba})). 

\begin{definition}\label{Def3} Let $S$ be a Hausdorff space, $C$ a subset of $S$ and $\Phi$ a family of real valued functions defined on $S$. We say that a point $x$ of $C$ is $\Phi$-exposed in $C$, and write $x\in \Phi\textnormal{Exp}(C)$, if there exists $\varphi\in \Phi$ such that $\varphi$ has a strict maximum on $C$ at $x$ i.e. $\varphi(x)>\varphi(y)$ for all $y\in C\setminus\lbrace x \rbrace$ (when C has at least two distinct points). Such $\varphi$ is then said to $\Phi$-expose $C$ at $x$. 
\end{definition} 
All the subsets considered in this article are assumed having at least two distinct points. The case of sets having only one point is trivial.  It is easy to see that $\Phi\textnormal{Exp}(C)\subset \Phi\textnormal{Ext}(C)$, but the converse is not true in general (see \cite{Ba} for more details).
\begin{definition} \textnormal{(see \cite{Kh})} Let $S$ be any nonempty set, $\Phi$ a familly of real-valued functions defined on $S$. Let $f$ be a real-valued function defined on $S$. We say that $f$ is $\Phi$-convex if and only if for every $a, x, y \in S$ such that $a$ is $\Phi$-between $x$ and $y$,
\begin{eqnarray*}
(f(x)\leq f(a) \textnormal{ and } f(y)\leq f(a))\Longrightarrow (f(x)=f(a)=f(y)). 
\end{eqnarray*}
\end{definition}
Let $S$ be a Hausdorff space and $\Phi$ a familly of real-valued functions defined on $S$. Let $K$ be a nonempty subset of $S$. By $\Phi\mathcal{C}(K)$ we denote the set of all real-valued $\Phi$-convex upper semi-continuous function on $K$, equipped with the following metric : $\forall f, g \in \Phi\mathcal{C}(K)$
$$\rho_{\infty}(f,g):=\sup_{x\in K} \frac{|f(x)-g(x)|}{1+|f(x)-g(x)|}.$$
The distance $\rho_{\infty}$ satisfies: for all $0< \varepsilon < 1$, 
$$\rho_{\infty}(f,g)\leq \varepsilon \Longleftrightarrow\sup_{x\in K}|f(x)-g(x)|\leq \frac{\varepsilon}{1-\varepsilon}.$$
\vskip5mm
A subspace $\mathcal{B}$ of $\Phi\mathcal{C}(K)$ is said to be $\Phi$-stable, if and only if 
$$(f, \varphi)\in \mathcal{B}\times \Phi \Longrightarrow f+\varphi \in\mathcal{B}.$$ 
Clearly, $\Phi$ is itself $\Phi$-stable if for example $\Phi$ is a vector space.
\vskip5mm
\subsection{Examples.} 
Let $X$ be a nonempty convex compact subset of some Hausdorff locally convex topological vector space $S$ and let $K$ be a nonempty closed subset of $X$ (not necessarily convex). We have the following propositions with $(\Phi,\|.\|_\Phi)=(\textnormal{Aff}(X),\|.\|_{\infty})$.
\begin{proposition} \label{equiv} \textnormal{(see also \cite[Proposition 2.]{Kh})} 
Let $x, y, z \in K$. Then, $x$ is $\textnormal{Aff}(X)$-between $y, z$ if and only if $x\in[y,z]$ (where $[y,z]$ denotes the segment in $X$). Consequentely, a point $x\in K$ is $\textnormal{Aff}(X)$-extremal in $K$ if and only if it is extremal in the classical sens.
\end{proposition}
\begin{proof} Suppose that $x\in[y,z]$, then there exists $\lambda \in [0,1]$ such that $x=\lambda y + (1-\lambda) z$. Let $\varphi\in \textnormal{Aff}(X)$ and suppose that $\varphi(y)\leq \varphi(x)$ and $\varphi(z)\leq \varphi(x)$. Then, $\varphi(y)= \varphi(x)=\varphi(z)$ and so, $x$ is $\textnormal{Aff}(X)$-between $y, z$. Indeed, suppose by contradiction that $\varphi(y)\neq \varphi(x)$ or $\varphi(z)\neq \varphi(x)$. Thus, $\varphi(y)<\varphi(x)$ or $\varphi(z)< \varphi(x)$. Since $\varphi$ is affine, 
\begin{eqnarray*}
\varphi(x)=\varphi(\lambda y + (1-\lambda) z)&=&\lambda \varphi(y)+(1-\lambda)\varphi(z)\\
                                             &<& \lambda \varphi(x)+(1-\lambda)\varphi(x)\\
                                             &=& \varphi(x),
\end{eqnarray*}
which is a contradiction. To see the converse, suppose that $x$ is $\textnormal{Aff}(X)$-between $y, z$. We need to prove that $x\in[y,z]$. Suppose by contradiction that $x\not\in[y,z]$. By the separation theorem, there exists $x^*$ in the topological dual of $S$ (in particular $x^*\in \textnormal{Aff}(X)$), such that $x^*(x)> x^*(\lambda y + (1-\lambda) z)$ for all $\lambda \in [0,1]$. In particular, $x^*(y)< x^*(x)$ and $x^*(z)< x^*(x)$, this implies that $x$ is not $\textnormal{Aff}(X)$-between $y, z$ which is a contradiction.
\end{proof}
The following proposition is easy to establish.
\begin{proposition} \label{pr2} 
The following assertions hold.

$(1)$ Let $f: X \longrightarrow \R$ be a convex function, then the restriction of $f$ to $K$, $f_{|K}: K \longrightarrow \R$ is $\textnormal{Aff}(X)$-convex. The set of all upper semi-continuous (resp. of all continuous) $\textnormal{Aff}(X)$-convex functions from $K$ into $\R$ is $\textnormal{Aff}(X)$-stable.

$(2)$ Let $f: X \longrightarrow \R$ be a strictly quasi-convex function, then the restriction of $f$ to $K$, $f_{|K}: K \longrightarrow \R$ is $\textnormal{Aff}(X)$-convex. But the space of all strictly quasi-convex functions from $X$ into $\R$ is not $\textnormal{Aff}(X)$-stable.

\end{proposition}

\subsection{The main result.}

By $(C(K),\|.\|_{\infty})$ we denote the Banach space of all real-valued continuous functions defined on a compact set $K$ and equipped with the sup-norm.
\vskip5mm

For reasons of completeness, we recall below the following simplified form of variational principle from \cite{Ba} that we will use.
\begin{lemma} \label{MBAD} \textnormal{(See \cite[Lemma 3.]{Ba})} Let $(K,d)$ be a compact metric space and $(\Phi,\|.\|_\Phi)$ be a Banach space included in $C(K)$ which separates the points of $K$ and such that $\alpha\|.\|_\Phi\geq \|.\|_{\infty}$ for some real number $\alpha> 0$. Let $f : (K,d) \rightarrow \R\cup \left\{+\infty \right\}$ be a proper lower semi-continuous function. Then, the set 
$$N(f)=\left\{\varphi  \in \Phi : f-\varphi \textnormal { does not have a unique minimum on } K \right\}$$
is of the first Baire category in $Y$. 
\end{lemma}

We give now our main result. If $K$ is compact and $f\in\Phi\mathcal{C}(K)$, by $K_{\max}(f)$, we denote the following closed subset of $K$:
$$\emptyset\neq K_{\max}(f):=\lbrace x\in K: f(x)=\max_K f\rbrace\subset K.$$
\begin{theorem} \label{Ba}  Let $(K,d)$ be a compact metric space and $(\Phi,\|.\|_\Phi)$ be a Banach space included in $C(K)$ which separates the points of $K$ and such that $\alpha\|.\|_\Phi \geq \|.\|_{\infty}$ for some real number $\alpha> 0$. Then, the following assertions hold.

$(1)$ Let $C\neq \emptyset$ be any closed subset of $K$, then $\emptyset\neq \Phi\textnormal{Exp}(C)\subset \Phi\textnormal{Ext}(C)$.

$(2)$ Let $I$ be any nonempty set and let $f_i \in \Phi\mathcal{C}(K)$ for all $i\in I$. Suppose that $\cap_{i\in I} K_{\max}(f_i)\neq \emptyset$. Then,
$$\Phi\textnormal{Exp}(\cap_{i\in I} K_{\max}(f_i))\cap \Phi\textnormal{Ext}(K)=\emptyset.$$
In particular, there exists a commun $\Phi$-extremal point $e$ of $K$ such that $f_i$ attains its maximum at $e$ for all $i\in I$. 

$(3)$ Let $\mathcal{B}$ be any $\Phi$-stable subspace of $(\Phi\mathcal{C}(K),\rho_{\infty})$. Then, generically, a function from $\mathcal{B}$ attains its maximum at a unique $\Phi$-extremal point of $K$. More precisely, the set 
$$\mathcal{G}:=\left\{f  \in \mathcal{B} : K_{\max}(f)\cap \Phi\textnormal{Ext}(K) \textnormal{ is a singleton } \right\}$$
is a $G_\delta$ dense subset of $(\mathcal{B},\rho_{\infty})$.
\end{theorem}
\begin{proof} $(1)$ Let $\delta_C : K\longrightarrow \R\cup\lbrace +\infty\rbrace$ be the function such that $\delta_C =0$ on $C$ and $\delta_C =+\infty$ on $K\setminus C$. This function is lower semi-continuous on $K$. From Lemma \ref{MBAD}, there exists a function $\varphi\in \Phi$ such that $\delta_C -\varphi$ has a unique minimum at some point $e\in K$. Equivalently, $\varphi$ has a unique maximum on $C$ attained at $e$, which implies that $\Phi\textnormal{Exp}(C)\neq \emptyset$. Now, it is easy to see that $\Phi\textnormal{Exp}(C)\subset \Phi\textnormal{Ext}(C)$. Indeed, let $e\in \Phi\textnormal{Exp}(C)$, then there exists $\varphi\in \Phi$ such that $\varphi(e)>\varphi(x)$ for all $x\in C\setminus \lbrace e \rbrace$. Let $y, z\in C$ such that $e$ is $\Phi$-between $y, z$. Since $\varphi(y), \varphi(z)\leq \varphi(e)$ and $e$ is $\Phi$-between $y, z$, it follows that $\varphi(y)=\varphi(e)=\varphi(z)$ which implies that $y=e=z$, since $e$ is the unique maximum of $\varphi$. Thus, $e$ is an $\Phi$-extremal point of $C$.
\vskip5mm
\noindent $(2)$ Since $f_i$ is upper semi-continuous then, the set $K_{\max}(f_i)$ is a compact subset of $K$. By hypothesis, $\cap_{i\in I} K_{\max}(f_i)\neq \emptyset$. Thus, $\cap_{i\in I} K_{\max}(f_i)$ is a nonempty compact subset of $K$. By part $(1)$, $$\emptyset\neq \Phi\textnormal{Exp}(\cap_{i\in I} K_{\max}(f_i))\subset \Phi\textnormal{Ext}(\cap_{i\in I} K_{\max}(f_i)).$$ Let $e\in \Phi\textnormal{Ext}(\cap_{i\in I} K_{\max}(f_i))$ and let us show that $e\in \Phi\textnormal{Ext}(K)$. Indeed, suppose that the contrary hold, that is there exists $y,z \in K$ such that $e$ is $\Phi$-between $y, z$ and $y\neq e$ or $z\neq e$. Since $e\in \Phi\textnormal{Ext}(\cap_{i\in I} K_{\max}(f_i))$ and $e$ is $\Phi$-between $y, z$, then either $y\in K\setminus \cap_{i\in I} K_{\max}(f_i)$ or $z\in K\setminus \cap_{i\in I} K_{\max}(f_i)$. We can assume without losing generality that is $y\in K\setminus \cap_{i\in I} K_{\max}(f_i)$. Thus, there exists $i_0\in I$ such that $y \not\in K_{\max}(f_{i_0})$. It follows that $e$ is $\Phi$-between $y, z$; $f_{i_0}(y)< \max_K f_{i_0} =f_{i_0}(e)$ and $f_{i_0}(z)\leq \max_K f_{i_0} =f_{i_0}(e)$. This contradicts the fact that $f_{i_0}$ is $\Phi$-convex. Thus, we have $$ \emptyset\neq \Phi\textnormal{Exp}(\cap_{i\in I} K_{\max}(f_i))\subset \Phi\textnormal{Ext}(\cap_{i\in I} K_{\max}(f_i))\subset \Phi\textnormal{Ext}(K).$$ 
\vskip5mm
\noindent $(3)$ For each $n\in \N^*$, let 
$$O_n:=\lbrace f\in \mathcal{B}; \exists x_n\in K/ f(x_n) > \sup\lbrace f(x): d(x,x_n)\geq \frac 1 n\rbrace\rbrace.$$
It is easy to see that $O_n$ is an open subset of $(\mathcal{B},\rho_{\infty})$ for each $n\in \N$. Thanks to Lemma \ref{MBAD}, for every $0<\varepsilon <1$ and every $f\in \mathcal{B}$, there exists a function $\varphi\in \Phi$ such that $\rho_{\infty}(\varphi,0)< \varepsilon$ and $-f -\varphi$ has a unique minimum on $K$ at some point $x_0$. This implies that $g:=f+ \varphi \in \cap_{n \in \N} O_n$ (we take $x_n=x_0$ for all $n\in \N$) and $\rho_{\infty}(g,f)< \varepsilon$. Thus $\cap_{n \in \N} O_n$ is dense in $(\mathcal{B},\rho_{\infty})$. It follows that $\cap_{n\in \N} O_n$ is a $G_\delta$ dense subset of $(\mathcal{B},\rho_{\infty})$. By following the idea of the proof of the variational principle of Devile-Godefroy-Zizler in \cite{DGZ}, we see that
$$\cap_{n\in \N} O_n = \left\{f  \in \mathcal{B} : f \textnormal{ has a unique maximum on } K \right\}.$$
Using part $(2)$, a unique maximum for a function from $\mathcal{B}$, is necessarily an extremal point. Hence, 
$$\cap_{n\in \N} O_n = \left\{f  \in \mathcal{B} : K_{\max}(f)\cap \Phi\textnormal{Ext}(K) \textnormal{ is a singleton } \right\}.$$
\end{proof}
\begin{corollary} Under the hypothesis of Theorem \ref{Ba}, for each $x\in K$, let $$\Omega_x:=\lbrace f\in \Phi\mathcal{C}(K): f(x)=\max_X f\rbrace.$$ Then, $\Omega_x$ is a nonempty closed cone subset of $(\Phi\mathcal{C}(K),\rho_{\infty})$ and there exists $e\in \Phi\textnormal{Ext}(K)$ such that $\Omega_x\subset\Omega_e$.
\end{corollary}
\begin{proof} It is easy to see that $\Omega_x$ is nonempty closed cone subset of $(\Phi\mathcal{C}(K),\rho_{\infty})$. From the definition of $\Omega_x$, we have that $x\in \cap_{f\in \Omega_x} K_{\max}(f)\neq\emptyset$. It follows from Theorem \ref{Ba} that $\cap_{f\in \Omega_x} K_{\max}(f)\cap \Phi\textnormal{Ext}(K)\neq \emptyset$ which gives the proof.

\end{proof}
\vskip5mm
Let $X$ be a nonempty convex compact subset of some Hausdorff locally convex topological vector space $S$. By considering the classe $(\Phi, \|.\|)=(\textnormal{Aff}(X),\|.\|_{\infty})$ and a compact subset $K$ of $X$ (not necessarily convex), we obtain from Theorem \ref{Ba} (using Proposition \ref{equiv} and Proposition \ref{pr2}) the following extension of the classical Bauer's maximum principle for compact metrizable sets. 
\begin{corollary} \label{Bauer} Let $X$ be a nonempty convex compact metrizable subset of some Hausdorff locally convex topological vector space $S$. Let $K$ be a any nonempty closed subset of $X$ (not necessarily convex). Then, the following assertions hold.

\noindent $(1)$ $\emptyset\neq \textnormal{AExp}(K)\subset\textnormal{Ext}(K)$.

\noindent $(2)$ Let $I$ be a nonempty set and $f_i: K\longrightarrow \R$ be a upper semi-continuous $\textnormal{Aff}(X)$-convex function, for all $i\in I$. Suppose that $\cap_{i\in I} K_{\max}(f_i)\neq \emptyset$, then $$\textnormal{AExp}(\cap_{i\in I} K_{\max}(f_i))\cap\textnormal{Ext}(K)\neq \emptyset.$$
In particular, there exists a commun extremal point $e$ at which $f_i$ attains its maximum on $K$, for each $i\in I$. 

$(3)$ Let $\mathcal{B}$ be a subset of all $\textnormal{Aff}(X)$-convex functions which is $\textnormal{Aff}(X)$-stable. Then, the set 
$$\mathcal{G}:=\left\{ f\in \mathcal{B}: K_{\max}(f)\cap \textnormal{Ext}(K)\textnormal{ is a singleton } \right\}$$
is a $G_\delta$ dense subset in $(\mathcal{B},\rho_{\infty})$. 
\end{corollary}
We can take for example, in the above corollary, $\mathcal{B}$ equal to the space of restrictions to $K$ of all upper semi-continuous (resp. of all continuous) convex functions from $X$ into $\R$. We can also take $K=X$ (convex in this case) and $\mathcal{B}$ equal to the space of all upper semi-continuous (resp. of all continuous) convex functions from $X$ into $\R$.
\vskip5mm
\paragraph{\bf Some classes of $\Phi$-convexity.} The result of this note applies in particular for the following ineresting classes of functions (these spaces do not satisfy the hypothesis of the Deville-Godefroy-Zizler variational principle \cite{DGZ}, \cite{DR}): 

$(1)$ $(\Phi,\|.\|_\Phi)=(\mathcal{H}(\overline{\Omega}),\|.\|_{\infty})$, where $\Omega$ is a bounded open connex subset of $\R^n$ and $\mathcal{H}(\overline{\Omega})$ is the space of all harmonique functions on $\Omega$ that are continuous on $\overline{\Omega}$, equipped with the sup-norm. 

$(2)$ $(\Phi,\|.\|_\Phi)=(\mathcal{P}_n^d(K),\|.\|_{\infty})$, where $K$ is a compact subset of $\R^n$ and $\mathcal{P}_n^d(K)$ is the set of all $n$-variable polynomial functions of degree less or equal to $d\geq 1$, equipped with the sup-norm.

$(3)$ $(\Phi,\|.\|_\Phi)=(Lip_0(K),\|.\|_{\textnormal{Lip}})$, where $(K,d)$ is a compact metric space and $Lip_0(K)$ is the space of all Lipschitz continuous functions that vanich at some point $x_0\in K$, equipped with the norm:
$$\|f\|_{\textnormal{Lip}}=\sup_{x,y\in K/x\neq y}\frac{|f(x)-f(y)|}{d(x,y)}.$$

$(4)$ $(\Phi,\|.\|_\Phi)=(\textnormal{Aff}(K),\|.\|_{\infty})$, where $K$ is convex compact metrizable set of some Hausdorff locally convex topological vector space and $\textnormal{Aff}(K)$ is the space of all real-valued affine continuous functions.

\bibliographystyle{amsplain}

\providecommand{\bysame}{\leavevmode\hbox to3em{\hrulefill}\thinspace}
\providecommand{\MR}{\relax\ifhmode\unskip\space\fi MR }
\providecommand{\MRhref}[2]{%
  \href{http://www.ams.org/mathscinet-getitem?mr=#1}{#2}
}
\providecommand{\href}[2]{#2}
\begin{thebibliography}{}

\end{thebibliography}


\begin{thebibliography}{999}
\bibitem{Ba} M. Bachir, \textit{On the Krein-Milman-Ky Fan theorem for convex compact metrizable sets}, Illinois Journal of Mathematics, 2018, 62 (1-2)
\bibitem{Bau} H. Bauer, Minimalstellen von Funktionen und Extremalpunkte II, Arch der Math., vol. 11, 1960, pp. 200-205.
\bibitem{CO} B. Cascales and J. Orihuela \textit{On Compactness in Locally Convex Spaces}, Math. Z. 195 (1987), 365-381.
\bibitem{DGZ} R. Deville and G. Godefroy and V. Zizler, \textit{A smooth variational principle with applications to Hamilton-Jacobi equations in infinite dimensions}, J. Funct. Anal. 111, (1993) 197-212.
\bibitem{DR} R. Deville, J. P. Revalski. \textit{Porosity of ill-posed problems}, Proc. Amer. Math. Soc. 128 (2000), 1117-1124.
\bibitem{Ky} K. Fan, \textit{On the Krein-Milman theorem}, Proc. Sympos. Pure Math., vol. 7, Amer. Math. Soc., (1963), 211-219.
\bibitem{Gr} M. W. Grossman, \textit{Relations of a paper of Ky Fan to a theorem of Krein-Milman type}, Alath. Z. 90 (1965), 212-214.
\bibitem{Kh} B. D. Khanh \textit{Sur la $\Phi$-Convexit\'e de Ky Fan}, J. Math. Anal. Appl. 20, (1967) 188-193.

\end{thebibliography}

\end{document}